\documentclass{article}
\usepackage{graphicx} 
\usepackage{bbm}
\usepackage{amsmath}
\usepackage{amssymb}
\usepackage{amsfonts}
\usepackage{amsthm}
\usepackage{physics}
\usepackage{comment}
\usepackage[margin=30truemm]{geometry}

\newcommand{\R}{\mathbb{R}}
\newcommand{\Z}{\mathbb{Z}}
\newcommand{\lam}{\lambda}
\newcommand{\E}{\mathcal{E}}

\newcommand{\suppm}{\mathrm{supp}\thinspace}
\newcommand{\tsquare}{\tilde{\square}}

\newcommand{\etabr}{\eta_{B_{R^2}}}
\newcommand{\F}{\mathcal{F}}

\theoremstyle{plain}
\newtheorem{thm}{Theorem}
\newtheorem*{thm*}{Theorem}
\newtheorem{prop}{Proposition}
\newtheorem{lem}{Lemma}
\newtheorem{rem}{Remark}

\theoremstyle{definition}
\newtheorem{dfn}{Definition}

\title{Local smoothing estimates for Schrödinger equations in modulation spaces}
\author{Kotaro Inami}
\date{}

\begin{document}

\maketitle
\begin{abstract}
    Motivated by a recent work of Schippa (2022), we consider local smoothing estimates for Schr\"{o}dinger equations in modulation spaces. By using the C\'{o}rdoba-Fefferman type reverse square function inequality and the bilinear Strichartz estimate, we can refine the summability exponent of modulation spaces. Next, we also discuss a new type of randomized Strichartz estimate in modulation spaces. Finally, we show that the reverse square function estimate implies the Strichartz estimates in modulation spaces. From this implication, we obtain the reverse square function estimate of critical order. 
\end{abstract}

\section{Introduction}
In this note, we consider the local smoothing property for the solution to the Schr\"{o}dinger equation
\begin{equation*}
    \begin{cases}
        i\frac{\partial u}{\partial t} + \Delta u = 0\\
        u(x,0) = f(x). 
    \end{cases}
\end{equation*}
For $f\in \mathcal{S}(\R^d)$, we write the solution to the above equation as
\begin{equation*}
    e^{it\Delta}f(x) := \frac{1}{(2\pi)^d}\int_{\R^d}e^{i(x\cdot \xi + t|\xi|^{2})}\widehat{f}(\xi)d\xi. 
\end{equation*}
The fixed time estimate for $e^{it\Delta}$ on $L^p$-based Sobolev spaces by Miyachi \cite{akihiko1981some} is as follows:
\begin{equation*}
    \|e^{it\Delta}f\|_{L^{p}_{x}(\R^d)} \leq c(1 + |t|)^{s}\|f\|_{W^{s,p}(\R^d)}\quad \forall s \geq d\left|\frac{1}{2} - \frac{1}{p}\right|,\thinspace \forall p\in (1,\infty), \forall t\in\R. 
\end{equation*}
In contrast, Rogers \cite{rogers2008local} showed the following local smoothing estimate:
\begin{equation*}
    \|e^{it\Delta}f\|_{L^{p}_{t,x}(I\times\R^{d})}\lesssim\|f\|_{W^{s,p}(\R^{d})}\quad \forall p> 2 + \frac{4}{d + 1}, \thinspace\forall s>2d\left(\frac{1}{2} - \frac{1}{p}\right) -\frac{2}{p}. 
\end{equation*}
where $I := [-1,1]$. Note that for this estimate, there is a $\frac{2}{p}$ derivative gain compared to the fixed-time estimate. After Rogers \cite{rogers2008local}, there is some progress on the local smoothing estimate for fractional Schr\"{o}dinger equations (for example, see \cite{gao2022improved, gao2022type, gan2022note}). 

In this note, we consider $L^{p}$ local smoothing estimates for the free Schrödinger propagator in modulation spaces: 
\begin{equation}
    \|e^{it\Delta}f\|_{L_{t}^{p}L_{x}^{q}(I \times \R^{d})} \lesssim \|f\|_{M^{s}_{r,t}(\R^d)}
    \label{eq:lpsmoothing}
\end{equation}
This type of estimate was first discussed by Schippa \cite{schippa2022smoothing}. He discussed this kind of inequality to investigate the well-posedness of nonlinear Schr\"{o}dinger equations with slowly decaying initial data. His result is as follows. 
\begin{thm}[Theorem 1.1 in {\cite{schippa2022smoothing}}]
    Suppose that $d \geq 1$, $2\leq p \leq \infty$, and $1 \leq q \leq \infty$. 
    Then the estimate 
    \begin{equation}
        \|e^{it\Delta}f\|_{L^{p}_{t,x}(I\times \R^d)} \lesssim \|f\|_{M^{s}_{p,q}(\R^d)} \label{eq:localsmoothing_schippa}
    \end{equation}
    holds if either one of the following conditions holds: 
    \begin{enumerate}
        \renewcommand{\theenumi}{\alph{enumi}}
        \renewcommand{\labelenumi}{(\theenumi)}
        
        \item If $2\leq p \leq \frac{2(d + 2)}{d}$, then, (\ref{eq:localsmoothing_schippa}) holds provided that $s > \mathrm{max}\{0, \frac{d}{2} - \frac{d}{q}\}$.

        \item If $\frac{2(d + 2)}{d}\leq p\leq \infty$ and $2\leq q\leq\infty$ , then, (\ref{eq:localsmoothing_schippa}) holds provided that $s > d - \frac{d + 2}{p} - \frac{d}{q}$.

        \item If $\frac{2(d + 2)}{d}\leq p\leq \infty$ and $1\leq q \leq\infty$, then (\ref{eq:localsmoothing_schippa}) is valid, provided that $s > 2(1 - \frac{1}{q})(\frac{d}{s} - \frac{d + 2}{p})$. 

        \item If $q = 1$, then (\ref{eq:localsmoothing_schippa}) holds with $s = 0$.

        \item If $d = 1$, $p = 4$, and $q = 2$, then (\ref{eq:localsmoothing_schippa}) holds with $s = 0$. 
    \end{enumerate}
    \label{thm:localsmpthingschippa}
\end{thm}

He also achieved the necessary conditions for (\ref{eq:lpsmoothing}). 
\begin{prop}[{Schippa \cite{schippa2022smoothing}}]
    Assume that (\ref{eq:lpsmoothing}) holds. Then it follows that 
    \begin{align}
    \begin{cases}
        d - \frac{d}{q} - \frac{2}{p} \leq \frac{d}{r} + s \\
        s \geq 0\\
        q\geq r
    \end{cases}
    \end{align}
    \label{prop:schippanecessary}
\end{prop}

The most essential point in these results is the case when
\begin{equation}
    \|e^{it\Delta}f\|_{L^{p_{d}}_{t,x}(I\times \R^d)} \lesssim \|f\|_{M^{\varepsilon}_{p_{d},2}(\R^d)}, \label{eq:schippa}
\end{equation}
where $p_{d} := \frac{2(d + 2)}{d},\thinspace\varepsilon > 0$. This case was shown via Bourgain-Demeter's $\ell^2$-decoupling estimate. Furthermore, by Proposition \ref{prop:schippanecessary}, we find this estimate optimal up to $\varepsilon$. After Schippa's work, Lu \cite{lu2023local} studied the local smoothing estimate for fractional Schr\"{o}dinger equations in $\alpha$-modulation spaces, and Chen-Guo-Shen-Yan \cite{chen2024smoothing} considered this type of smoothing estimate on the cylinder $\mathbb{T}^{n}\times\R^{m}$. 

The aim of this paper is to further explore local smoothing estimates and Strichartz-type estimates in modulation spaces. Our first result concerns a refinement of the summability exponent in the modulation norm. 
\begin{thm}
    For any $\varepsilon > 0$, $q\in [1, 4)$, the estimate 
    \begin{equation}
        \|e^{it\Delta}f\|_{L^{4}_{t,x}(I\times\R)}\lesssim_{\varepsilon} \|f\|_{M^{\varepsilon}_{2,q}(\R)}
    \end{equation}
    holds. \label{thm:loocalsmoothing}
\end{thm}

In previous works (\cite{schippa2022smoothing}, \cite{lu2023local}, and \cite{chen2024smoothing}), the $\ell^{2}$-decoupling estimate played an essential role in the proof. In contrast, our proof of Theorem \ref{thm:loocalsmoothing} is based on the C\'{o}rdoba-Fefferman type reverse square functions estimate.  

\begin{rem}
    According to Proposition \ref{prop:schippanecessary}, if the estimate
    \begin{equation*}
        \|e^{it\Delta}f\|_{L^{4}_{t,x}(I\times\R)} \lesssim \|f\|_{M^{s}_{p,q}(\R)}
    \end{equation*}
    holds, then we have $q \le 4$. Thus, Theorem \ref{thm:loocalsmoothing} is almost sharp in this sense. However, by Proposition \ref{prop:schippanecessary}, we also have $p\leq 4$. Hence, Theorem \ref{thm:loocalsmoothing} could be improved and the estimate
    \begin{equation*}
        \|e^{it\Delta}f\|_{L^{4}_{t,x}(I\times\R)}\lesssim_{\varepsilon} \|f\|_{M_{4,4}(\R)}. 
    \end{equation*}
    might be true.
\end{rem}

Our estimates can be applied to the study of nonlinear problems. For instance, we can show the local well-posedness result for the cubic nonlinear Schr\"{o}dinger equation:
\begin{align}
    \begin{cases}
        i\frac{\partial u}{\partial t} + \Delta u = \pm |u|^2 u\\
        u(\cdot,0) = f(\cdot). \label{eq:cubicNLS}
    \end{cases}
\end{align}
Interpolating our estimate with 
\begin{equation*}
    \|e^{it\Delta}f(x)\|_{L^{4}_{t,x}(I\times \R)} \lesssim \|f\|_{M^{\frac{1}{4}}_{4,4}(\R)}
\end{equation*}
in Theorem \ref{thm:localsmpthingschippa}, we also have
\begin{equation*}
    \|e^{it\Delta}f(x)\|_{L^{4}_{t,x}(I\times \R)} \lesssim \|f\|_{M^{s}_{p,q}(\R)}
\end{equation*}
for $2\le p \le 4$, $2\le q < 4$, and $s > \frac{1}{4}$. Combining this estimate with the inhomogeneous Strichartz estimate
\begin{equation*}
    \left\|\int^{t}_{0}e^{i(t -s)\Delta}F(s)ds\right\|_{L^{8}_{t}L^{4}_{x}([0,T]\times \R)}\lesssim \|F\|_{L^{\frac{8}{7}}_{t} L^{\frac{4}{3}}_{x}([0,T]\times\R)}, 
\end{equation*}
and applying the standard iteration argument, we obtain the local well-posedness for \eqref{eq:cubicNLS} in $M^{s}_{p,q}(\R)$ with $2\le p \le 4$, $2\le q < 4$, $s > \tfrac{1}{4}$. This result provides a partial refinement of Theorem 1 in \cite{klaus2023}.

We also prove a time-global version of this estimate for randomized initial data. In this paper, we consider Wiener's randomization. Wiener's randomization and assumption (\ref{assumption:randomization}) will be explained in the next section. 
\begin{thm}
    Given $f \in \mathcal{S}'(\R^d)$, let $f^{(\omega)}$ be the Wiener randomization of $f$ with the assumption (\ref{assumption:randomization}) and let $(p,q)\in[1,\infty)^{2}$ satisfy
    \begin{equation*}
        \frac{1}{p} = \frac{d}{2}\left(\frac{1}{2} - \frac{1}{q}\right). 
    \end{equation*}
    If $2 \leq q < \frac{2(d +1)}{d - 1}$, then for any $\varepsilon > 0$, the estimate \begin{equation}
        \|e^{it\Delta}f^{(\omega)}\|_{L^{p}_{t}L^{q}_{x}(\R^{d + 1})} \lesssim \left(\log\frac{1}{\varepsilon} + 1\right)\|f\|_{M_{2,\frac{4q}{q + 2}}(\R^d)} \label{eq:randomrefinedstr}
    \end{equation}
    holds with probability at least $1 - \varepsilon$. \label{thm:randomized}
\end{thm}
The randomization technique for nonlinear PDE was first introduced by Bourgain \cite{bourgain1996invariant}. He treated nonlinear Schr\"{o}dinger equations in 2-dimensional torus. After that, Burq and Tzvetkov studied nonlinear wave equations with randomized initial data in \cite{burq2008random} and \cite{burq2008random2}. In addition to these papers, many authors have studied nonlinear PDE with randomized initial data setting. 

Improved Strichartz estimates for Schr\"{o}dinger equations with randomized initial data were first discussed by B\'{e}nyi-Oh-Pocovnicu \cite{benyi2015wiener}. The Strichartz estimate of modulation spaces are usually known as the following form (for example, see Proposition 5.1 in \cite{wang2007global}):
\begin{equation*}
    \|\ev{k}^{\alpha}\|e^{it\Delta}\square_{k}f\|_{L^{p}_{t}L^{q}_{x}(\R\times\R^{d})}\|_{\ell^{\beta}_{k}} \lesssim \|f\|_{M^{\gamma}_{r,u}(\R^{d})}. 
\end{equation*}
According to Remark 2.4. in \cite{benyi2015wiener}, if we consider this type of Strichartz estimate, there is no advantage of randomization. However, in our setting, by using the orthogonal Strichartz estimate, we gain refinement in the summable exponent in the modulation norm. This exponent expresses how fast the Fourier transform of the function decays at infinity. Thus, improvement in the summability exponent could be regarded as an improvement in the regularity of initial data in this sense. 

As a by-product of the discussion about local smoothing estimates, we show the sharpness of the following reverse square function estimate for slabs. 
\begin{prop}[Exercise 5.30 {\cite{demeter2020fourier}}]
    Let $d \geq 1$ and $R\geq 1$. Suppose that 
    \begin{equation*}
    \mathrm{supp}\thinspace\widehat{F}\subset \{(\xi, |\xi|^{2} + \tau)\in \R^{d}\times\R \thinspace;\thinspace |\xi|\leq 1,\thinspace|\tau| \leq R^{-2}\}. 
    \end{equation*}
    If $p \geq \frac{2(d +2)}{d}$, then we have 
    \begin{equation}
        \|F\|_{L^{p}(\R^{d +1})} \lesssim_{\varepsilon} R^{\frac{d}{2} - \frac{d + 1}{p} + \varepsilon}\|(\sum_{k\in\Z^{d}}|F_{\square_{k}}|^{2})^{\frac{1}{2}}\|_{L^{p}(\R^{d +1})} \label{prop:cfdemeter}
    \end{equation}
    where $F_{\square_{k}} := \F^{-1}[\chi_{[0,1]^{d}\times\R}(R(\cdot - k))\widehat{F}]$. \label{prop:demeter}
\end{prop}

Estimate (\ref{prop:cfdemeter}) with $p\geq \frac{2(d +2)}{d}$ (resp. $R^{\frac{d}{2} - \frac{d + 1}{p}+\varepsilon}$) is replaced by $p\geq \frac{2(d +1)}{d}$ (resp. $R^{\varepsilon}$) is known as the reverse square function conjecture. This conjecture remains open except for the case $(d, p)=(1,4)$. In this article, we focus on the supercritical case $p= \frac{2(d +2)}{d}$. We show the sharpness of this inequality when $p = \frac{2(d + 2)}{d}$ using the necessary conditions for the Strichartz-type estimate in modulation spaces. 
\begin{prop}
    Let $d\geq 1$, $R\geq 1$, $s > 0$, and $p = \frac{2(d + 2)}{d}$. Assume that the inequality
    \begin{equation}
        \|F\|_{L^{p}(\R^{d + 1})} \lesssim R^{s}\|(\sum_{k\in\Z^{d}}|F_{\square_{k}}|^{2})^{\frac{1}{2}}\|_{L^{p}(\R^{d + 1})}\label{eq:asmnecessary}
    \end{equation}
    holds for $F\in \mathcal{S}(\R^d)$ with $\suppm \widehat{F}\subset \{(\xi, |\xi|^{2} + \tau)\in \R^{d}\times\R \thinspace;\thinspace |\xi|\leq 1,\thinspace|\tau| \leq R^{-2}\}$. Then $s \geq \frac{d}{2} - \frac{d +1}{p}$. \label{prop:necessary for Demeter}
\end{prop}
Gan \cite{gan2024small} showed the sharpness of Proposition \ref{prop:demeter} when $d = 1$ and $p \geq 2$. On the other hand, we treat the multidimensional and $p = \frac{2(d + 2)}{d}$ cases.

\section{Preliminaries}
\subsection{Modulation spaces}
We briefly recall modulation spaces. For further details and applications, for example, see \cite{Benyi-Okoudiou}, \cite{cordero-rodino}, \cite{feichtinger1983modulation}, \cite{grochenig2001foundations}, and \cite{wang_harmonic}. For the definition of modulation spaces, we introduce the following Fourier multiplier:
\begin{equation*}
    \widehat{\square_{k}f}(\xi):= (\varphi(D - k)f)\thinspace\widehat{}\thinspace(\xi) := \varphi(\xi - k)\widehat{f}(\xi), \quad \xi\in\R^d
\end{equation*}
where $\varphi\in C^{\infty}_{0}(\R^d)$ with $\mathrm{supp}\thinspace\varphi \subset [-1,1]^d$ and ${\displaystyle\sum_{k\in\Z^{d}}}\varphi(\xi-k) = 1$ for all $\xi\in\R^d$. The (quasi) norm of modulation spaces is defined for any $f\in\mathcal{S}'(\R^d)$, $s\in\R$, and $p,q\in (0,\infty]$ as follows:
\begin{equation*}
    \|f\|_{M^{s}_{p,q}(\R^d)} := \left\|\ev{k}^{s}\|\square_{k}f(x)\|_{L^{p}_{x}(\R^d)}\right\|_{\ell^{q}_{k}(\Z^d)}
\end{equation*}
where $\ev{k} := (1 + |k|^{2})^{\frac{1}{2}}$. For $s\in\R$, $p,q\in(0,\infty]$, we define modulation spaces as follows: 
\begin{align*}
    M^{s}_{p,q}(\R^{d}) &:= \{f\in\mathcal{S}'(\R^{d}\thinspace;\thinspace \|f\|_{M^{s}_{p,q}} < \infty\}, \\
    M_{p,q}(\R^d) &:= M^{0}_{p,q}(\R^d). 
\end{align*}
Note that modulation spaces are quasi-Banach spaces and do not depend on the choice of window function $\varphi$. 

In the proof of the main theorem, we use the boundedness of the Schr\"{o}dinger propagators on modulation spaces. This is established in \cite{benyi2007unimodular}. 
\begin{prop}
    Let $1 \leq p,q \leq \infty$ and $t \in\R$. Then the inequality 
    \begin{equation*}
        \|e^{it\Delta}f\|_{M_{p,q}(\R^d)} \lesssim (1 + |t|^{2})^{\frac{d}{2}}\|f\|_{M_{p,q}(\R^d)}
    \end{equation*}
    holds for all $f\in\mathcal{S}'(\R^d)$. 
    \label{prop:unimodular}
\end{prop}

The Littlewood-Paley characterization for modulation spaces is useful.
\begin{prop}[{\cite[Theorem 1]{chaichenets2020local}}]
    Let $d\in\mathbb{N}$, $p,q\in [1,\infty]$, $s\in\R$. Then \begin{equation*}
        \|f\|_{M^{s}_{p,q}(\R^d)} \sim \|2^{sj}\|\Delta_{j}f\|_{M_{p,q}(\R^d)}\|_{\ell^{q}_{j}}
    \end{equation*}
    where $\{\Delta_{j}\}$ denotes the Littlewood-Paley decomposition. \label{thm:littlewoodpaley}
\end{prop}

\subsection{Wiener's randomization}
We also prove space-time estimates for randomized initial data. It is well-known that 
\begin{dfn}
    Let $\{g_{k}\}_{k\in\mathbb{Z}^d}$ be a sequence of complex-valued independent random variables on a probability space $(\Omega, \mathfrak{F}, P)$, where the real parts and imaginary parts of $g_{k}$ are independent. Then, we define the Wiener randomization of $f\in\mathcal{S}'(\R^d)$ as
    \begin{equation}
        f^{\omega}:= \sum_{k\in\mathbb{Z}^d}g_{k}(\omega)\varphi(D-k)f
    \end{equation}
    where the Fourier multipliers $\varphi(D - k)$ are defined in the definition of modulation spaces. 
\end{dfn}

Let $\mu^{(1)}_{k}$ and $\mu^{(2)}_{k}$ denote the distributions of $\mathrm{Re}\thinspace g_{k}$ and $\mathrm{Im}\thinspace g_{k}$, respectively. We will also make the following additional assumption: there exists $c > 0$ such that 
\begin{equation}
    \left|\int_{\R^d}e^{\gamma x}d\mu^{(i)}_{k}\right| \leq e^{c\gamma^2}\label{assumption:randomization}
\end{equation}
for all $\gamma \in \R$, $k\in\Z^d$, $i = 1,2$.

\subsection{Strichartz estimate for orthogonal initial data.}
In the proof for the main theorems, we will use the Strichartz estimate for orthonormal initial data. 
The following Strichartz estimates for orthonormal initial data were established by Frank-Lewin-Lieb-Seiringer \cite{frank2014strichartz} and Frank-Sabin \cite{frank2017restriction}:
\begin{thm}
    Let $(p,q)\in [1,\infty)^{2}$ satisfy \begin{equation}
        \frac{1}{p} = \frac{d}{2}\left(\frac{1}{2} - \frac{1}{q}\right). 
    \end{equation}
     If $2 \leq q \leq \frac{2(d+1)}{d-1}$ and $\beta \leq \frac{2q}{q+2}$, then, 
    \begin{equation}
        \left\|\sum^{\infty}_{k=1}\nu_{k}|e^{it\Delta}f_{k}|^2\right\|_{L^{\frac{p}{2}}_{t}L^{\frac{q}{2}}_{x}(\R^{d + 1})} \lesssim \|\nu\|_{\ell^{\beta}}.\label{eq:ons}
    \end{equation}
    holds for any orthonormal systems $(f_{k})^{\infty}_{k = 1}$ in $L^{2}(\R^d)$ and for any $\{\nu_{k}\}\subset \ell^{\beta}$. This estimate is sharp in the sense that the estimate fails if $\beta > \frac{2q}{q+2}$.
    
\end{thm}
Note that the inequality (\ref{eq:ons}) is equivalent to the following
\begin{equation}
    \left\|\left( \sum^{\infty}_{k = 1}|e^{it\Delta}f_{k}(x)|^{2}\right)^{\frac{1}{2}}\right\|_{L^p_{t}L^q_{x}(\R^{d + 1})} \lesssim \left\|\|f_{k}\|_{L^{2}_{x}(\R^d)}\right\|_{\ell^{2\beta}_{k}}
    \label{eq:ogs}
\end{equation}
where $(f_{k})^{\infty}_{k = 1}\subset L^{2}(\R^d)$ is an orthogonal system. Using this estimate (\ref{eq:ogs}), Frank and Sabin \cite{frank2017restriction} improved the Strichartz estimates for a single initial datum in Besov spaces. We will employ their idea to show the refined Strichartz estimates in modulation spaces.

\subsection{Notations}
\begin{itemize}
    \item $\widehat{f}(\xi) := \int_{\R^d} e^{i\xi \cdot x}f(x) dx\quad (\xi\in\R^d)$. 
    \item $A\lesssim B$~(resp. $A\gtrsim B$ ) denotes that there exists $C>0$ such that $A \leq CB$~(resp. $A \geq CB$). If $A\lesssim B$ and $A\gtrsim B$ hold, we write $A \sim B$. 
    \item $B_{d}(a,R)$ denotes the $d$-dimensional ball centered at $a\in\R^{d}$ with radius $R$. 

    \item For a function $f$, the symbol $\|f\|_{L^{p}_{y}}$ denotes the $L^{p}$ norm with respect to $y$. 

    \item For a weight $\eta:\R^d \to [0,\infty)$, we define $\|f\|_{L^p(\eta)} := \big(\int_{\mathbb{R}^d} |f(x)|^p \eta(x)\,dx\big)^{1/p}$. We use the weight adapted to $B_{d}(a,R)$:
    \begin{equation*}
        w_{B_{d}(a,R)}(x) := \left(1 + \frac{|x - a|}{R}\right)^{-100d}.
    \end{equation*}

    \item Let $\phi\in C^{\infty}_{0}(\R^{d})$ satisfy $\phi = 1$ on $[0,1]^{d}$ and $\suppm \phi\subset[-\frac{1}{2}, \frac{3}{2}]^{d}$. We use $f_{\square_{k}}$ for the smooth projection adapted to the cube $[R^{-1}k,R^{-1}(k+1)]^{d}$ i.e. 
\begin{equation*}
    f_{\square_{k}}:= \F^{-1}[\phi(R(\cdot - k))\widehat{f}]. 
\end{equation*}
On the other hand, $f_{\widehat{\square}_{k}}$ denotes non-smooth projections:
\begin{equation*}
    f_{\widehat{\square}_{k}}:= \F^{-1}[\chi_{[0,1]^{d}}(R(\cdot - k))\widehat{f}]. 
    \end{equation*}
    As we defined in the definition of modulation spaces, we also use the notation $\square_{k}f$ that denotes the projection adapted to the cube $[k-1, k+ 1]^d$: 
    
    \item $\mathcal{N}_{\delta}(\mathbb{P}^{d})$ denotes the $\delta$-neighborhood of d-dimensional truncated paraboloid:
\begin{equation*}
    \mathcal{N}_{\delta}(\mathbb{P}^{d}) := \{(\xi, |\xi|^{2} +\tau)\in\R^{d + 1} ; |\xi|\leq 1, \thinspace |\tau| \leq \delta\}. 
\end{equation*}

    \item We define the extension operator $\E$ associated to the paraboloid:
\begin{equation*}
    \E f(t,x) := \int_{[0,1]^{d}}e^{i(x\cdot\xi + t|\xi|^{2})}\widehat{f}(\xi)d\xi
\end{equation*}
where $\suppm\widehat{f}\subset B_{d}(0,1)$. 

\item The Littlewood–Paley operators $\{\Delta_j\}_{j\in\mathbb{N}}$ are defined as follows.  
Let $\phi \in C_c^\infty(\mathbb{R}^d)$ be a radial bump function such that 
$\phi(\xi) = 1$ for $|\xi|\le 1$ and $\phi(\xi)=0$ for $|\xi|\ge 2$.  
Set $\psi(\xi) := \phi(\xi) - \phi(2\xi)$.  
Then we define
\begin{equation*}
    \Delta_0 f := \mathcal{F}^{-1}[\phi \, \hat{f}], 
\end{equation*}
and for $j \ge 1$,
\begin{equation*}
    \Delta_j f := \mathcal{F}^{-1}[\psi(2^{-j}\cdot)\, \hat{f}].
\end{equation*}
In particular, $\sum_{j\ge 0}\Delta_j f = f$ in $\mathcal{S}'(\mathbb{R}^d)$.
\end{itemize}

\section{Proof of Theorem \ref{thm:loocalsmoothing}}
\subsection{Proof of the reverse square function estimate}
 To prove Theorem \ref{thm:loocalsmoothing}, we need the reverse square function estimate. 
\begin{prop}
    Let $d = 1$, $f\in \mathcal{S}(\R)$ with $\mathrm{supp}\widehat{f}\subset B_{1}(0,1)$, and $R\geq 1$. Then for each ball $B_{R^2}$ with radius $R^2$, we have
    \begin{equation}
        \|\E f\|_{L^{4}_{t,x}(B_{R^2})} \lesssim\|(\sum_{k\in\Z}|\E f_{\square_{k}}|^{2})^{\frac{1}{2}}\|_{L^{4}_{t,x} (w_{B_{R^2}})}\label{eq:cordoba-fefferman}. 
    \end{equation}
    \label{prop:cordoba-fefferman type est}
\end{prop}

\begin{rem}
    If $\square_{k}$'s are replaced by non-smooth projections $\widehat{\square}_{k}$, this proposition is a consequence of Theorem 1.2 in \cite{gressman2021reversing}. Using the boundedness of the vector-valued Hilbert transform, we can deduce the smooth version from the non-smooth version. However, here we will provide another proof of this proposition. Moreover, we will also provide a direct proof of this proposition in Appendix A.
\end{rem}

Here we recall the classical C\'{o}rdoba-Fefferman estimate (for example, see \cite{cordoba1979}, \cite{fefferman}, and Chapter 3 in \cite{demeter2020fourier}). 
\begin{prop}
    Let $R\geq 1$. For any $F\in\mathcal{S}(\R^2)$ with $\suppm \widehat{F}\subset \mathcal{N}_{R^{-2}}(\mathbb{P}^{1})$. Then, we have
    \begin{equation*}
        \|F\|_{L^{4}(\R^{2})} \lesssim \|(\sum_{k\in\Z}|F_{\widehat{\square}_{k}}|^{2})^{\frac{1}{2}}\|_{L^{4}(\R^2)}
    \end{equation*}
    \label{prop:classical cordoba fefferman}
\end{prop}
We observe that the classical C\'{o}rdoba-Fefferman's estimate implies a reverse square function estimate for the extension operator. Combining the following lemma and Proposition \ref{prop:classical cordoba fefferman}, we obtain Proposition \ref{prop:cordoba-fefferman type est}. 
\begin{lem}
    Let $p\in [2,\infty)$, $s\geq 0$, and $R\geq 1$.  Suppose that the inequality
    \begin{equation}
        \|F\|_{L^{p}(\R^{d + 1})} \lesssim R^s \|(\sum_{k\in\Z^{d}}|F_{\widehat{\square}_{k}}|^{2})^{\frac{1}{2}}\|_{L^{p}(\R^{d+1})}\label{asm:classicalc-f}
    \end{equation}
    holds for any $F\in\mathcal{S}(\R^{d + 1})$ with $\suppm \widehat{F}\subset \mathcal{N}_{R^{-2}}(\mathbb{P}^{1})$. Then, for any ball $B_{R^{2}}\subset\R^{d+1}$ and for any $f\in\mathcal{S}(\R^{d})$ with $\suppm \widehat{f}\subset B_{d}(0,1)$, the estimate
    \begin{equation*}
        \|\E f\|_{L^{p}(B_{R^2})} \lesssim R^{s}\|(\sum_{k\in\Z^{d}}|\E f_{\square_{k}}|^{2})^{\frac{1}{2}}\|_{L^{p}(w_{B_{R^{2}}})}
    \end{equation*}
    holds. 
    \label{lem:sufficientpart}
\end{lem}
\begin{proof}
    Let $\etabr\in \mathcal{S}(\R^{d+1})$ satisfy
    \begin{align*}
        \begin{cases}
            \suppm\F[\etabr^{\frac{1}{p}}] \subset B_{d +1}(0,R^{-2})\\
            \etabr \gtrsim 1 \thinspace \mathrm{on}\thinspace B_{d + 1}(0,R^{2})
        \end{cases}
    \end{align*}
    (This function is constructed by using a bump function $\theta$ whose Fourier support is contained in $B_{2}(0,1/2)$ and that satisfies $\ge 1$ on $B_{s}(0,1)$ in the physical space. Then one may consider $(\theta \cdot \overline{\theta})^{4}$ and apply a suitable rescaling). 
    We use the notation $\mathcal{F}_{t,x}$ for the Fourier transform with respect to $t$ and $x$. Note that $\suppm \F_{t,x}[\E f\cdot \etabr^{\frac{1}{p}}]$ is contained in a $R^{-2}$-neighborhood of the compact paraboloid. Thus, we can apply (\ref{asm:classicalc-f}) to this. It follows that
    \begin{equation}
        \|\E f\cdot \etabr^{\frac{1}{p}}\|_{L^{p}_{t,x}(\R^{d + 1})} \lesssim_{\varepsilon} R^{s}\|(\sum_{k\in\Z^{d}}|(\E f\cdot \etabr^{\frac{1}{p}})_{\widehat{\square}_{k}}|^{2})^{\frac{1}{2}}\|_{L_{t,x}^{p}(\R^{d+1})}.\label{eq:forprop1} 
    \end{equation}
    Let $\square_{k} '$ denote the sum of $\square_{k}$ and the squares adjacent to $\square_{k}$ (i.e. $\square'_{k} = \bigcup_{\substack{m\in \mathbb{Z}^d \\ \| m-k\|_{\ell^{\infty}} \le 1}}\square_{m}$). Then \begin{equation*}
        (\E f\etabr^{\frac{1}{p}})_{\square_{k}} = (\E f_{\square_{k} '}\etabr^{\frac{1}{p}})_{\square_{k}}
    \end{equation*}
    holds. Therefore, from the boundedness of the vector-valued Hilbert transform (for example, see Section 4.5 in \cite{grafakos_modern}), we obtain
    \begin{align*}
        \|(\sum_{k\in\Z^{d}}|(\E f\cdot \etabr^{\frac{1}{p}})_{\square_{k}}|^{2})^{\frac{1}{2}}\|_{L_{t,x}^{p}(\R^{d+1})} &= \|(\sum_{k\in\Z^{d}}|(\E f_{\square_{k} '}\cdot \etabr^{\frac{1}{p}})_{\square_{k}}|^{2})^{\frac{1}{2}}\|_{L_{t,x}^{p}(\R^{d+1})}\\
        &\lesssim  \|(\sum_{k\in\Z^{d}}|\E f_{\square_{k}}\cdot \etabr^{\frac{1}{p}}|^{2})^{\frac{1}{2}}\|_{L_{t,x}^{p}(\R^{d+1})}.
    \end{align*}
    Combining this and (\ref{eq:forprop1}), we get the desired result.   
\end{proof}

\subsection{Proof of the bilinear Strichartz estimate}

We will also use a kind of bilinear Strichartz estimate in the proof of Theorem \ref{thm:loocalsmoothing}. Here, we prove the estimate below by following Tao's argument in \cite{tao2010physical}. 
\begin{prop}
    Let $k_{1}, k_{2}\in \Z$ be with $|k_{1} - k_{2}| \geq 3$. For $f,g\in \mathcal{S}(\R)$ with $\suppm\widehat{f}\subset [k_{1}-1, k_{1}+1]$ and $\suppm\widehat{g}\subset [k_{2}-1, k_{2} + 1]$, we have the following inequality:
    \begin{equation*}
        \|[e^{it\Delta}f][e^{it\Delta}g]\|_{L^{2}_{t,x}(\R\times\R)}\lesssim \ev{k_{1} - k_{2}}^{-\frac{1}{2}}\|f\|_{L^{2}_{x}(\R)}\|g\|_{L^{2}_{x}(\R)}. \label{thm:bilinear}
    \end{equation*}
\end{prop}
To show this, we use the following local smoothing estimate.
\begin{prop}[{\cite[Corollary 3.4]{tao2010physical}}]
    Let $d\geq 1$, let $\omega$ be a unit vector in $\R^d$, and let $f\in \mathcal{S}(\R^d)$ with $\suppm\widehat{f}\subset\{\xi\in\R^{d}\thinspace;\thinspace |\xi\cdot\omega| \sim N\}$ for some $N >0$. Then \begin{equation}
        (\int_{\R}\int_{x\cdot\omega = 0}|e^{it\Delta}f(x)|^{2}dxdt)^{\frac{1}{2}} \lesssim N^{-\frac{1}{2}}\|f\|_{L^{2}_{x}(\R^d)}
    \end{equation}
    where $dx$ is the $(d-1)$-dimensional Lebesgue measure on the hyperplane $\{x\in\R^d \thinspace;\thinspace x\cdot\omega = 0\}$. \label{thm:taosmoothing}
\end{prop}

\begin{proof}[Proof of Proposition \ref{thm:bilinear}]
    Define the function $U(t,x,y)\thinspace:\thinspace\R\times\R\times\R\rightarrow\R$ as follows:
    \begin{equation*}
        U(t,x,y) := [e^{it\Delta}f(x)][e^{it\Delta}g(y)]. 
    \end{equation*}
    This solves the $2$-dimensional Schrödinger equation
    \begin{equation*}
        iU_{t} + \Delta_{x,y}U = 0
    \end{equation*}
    with an initial datum whose Fourier transform is supported in the region
    \begin{equation*}
        \{(\xi_{1},\xi_{2})\in\R^2\thinspace;\thinspace k_{1}-1\leq \xi_{1}\leq k_{1} + 1,\thinspace k_{2}-1\leq \xi_{2}\leq k_{2} + 1\}. 
    \end{equation*}
    Set a unit vector $\omega := (\frac{1}{\sqrt{2}}, -\frac{1}{\sqrt{2}})$. Taking the inner product $\omega$ and $(\xi_{1}, \xi_{2})\in\R^{2}$, we have
    \begin{equation*}
        |(\xi_{1}, \xi_{2})\cdot\omega| = \frac{1}{\sqrt{2}}|\xi_{1} - \xi_{2}|. 
    \end{equation*}
    We may assume that $k_{1} \geq k_{2}$. Then we have the following:
    \begin{equation*}
        k_{1}-k_{2} - 2 \leq |(\xi_{1}, \xi_{2})\cdot\omega| \leq k_{1}-k_{2} + 2. 
    \end{equation*}
    Since $|k_{1} -k_{2}| \geq 3$, we have $|(\xi_{1}, \xi_{2})\cdot\omega|\sim\ev{k_{1}-k_{2}}$. Applying Proposition \ref{thm:taosmoothing} to $U(t,x,y)$, we get
    \begin{equation*}
        \|U(t,x,x)\|_{L^{2}_{t,x}(\R\times\R)} \lesssim \ev{k_{1} -k_{2}}^{-\frac{1}{2}}\|U(0,x,y)\|_{L^{2}_{x,y}(\R^2)}. 
    \end{equation*}
    This inequality coincides with the desired one. 
\end{proof}

\subsection{Proof of Theorem \ref{thm:loocalsmoothing}}
The key lemma to show the main result is the following:
\begin{lem}
    Let $\lam \geq 1$. For any $\varepsilon \in (0,1)$ and for any $f\in \mathcal{S}'$ with $\suppm\widehat{f}\subset\{\xi\in\R\thinspace ; \thinspace |\xi| \leq \lam\}$, the following inequality holds:
    \begin{equation}
        \|e^{it\Delta}f\|_{L^{4}_{t,x}(I\times\R)} \lesssim \|f\|_{M_{2,4-\varepsilon}(\R)}. \label{eq:presmoothing}
    \end{equation}
    \label{lem:mainprooflemma}
\end{lem}

\begin{proof}
    Rescaling $f$, we have the following.
    \begin{equation*}
        \|e^{it\Delta}f\|_{L^{p}_{t,x}(I\times\R^{d})} = \lam^{-\frac{d + 2}{p}}\|e^{it\Delta}g\|_{L^{p}_{t,x}(B_{\lam^{2}}(0)\times\R^{d})}
    \end{equation*}
    where $g(x) := \lam^{-d}f(\lam^{-1}x)$. We decompose $\R^d$ into finitely overlapping balls $B_{\lam^{2}}$ with radius $\lam^2$, that is, 
    \begin{equation*}
        \R^d = \bigcup B_{\lam^2}
    \end{equation*}
    Then we have
    \begin{equation*}
        \|e^{it\Delta}g\|_{L^{4}(B_{1}(0,\lam^{2})\times\R)}\leq (\sum_{B_{\lam^{2}}}\|e^{it\Delta}g\|^{4}_{L^{4}_{t,x}(B_{1}(0,\lambda^{2})\times B_{\lam^{2}})})^{\frac{1}{4}}. 
    \end{equation*}
    Applying the C\'{o}rdoba-Fefferman type inequality (\ref{eq:cordoba-fefferman}), we obtain the following:
    \begin{equation*}
        (\sum_{B_{\lam^{2}}}\|e^{it\Delta}g\|^{4}_{L^{4}_{t,x}(B_{1}(0,\lambda^{2})\times B{\lam^{2}}}))^{\frac{1}{4}} \lesssim (\sum_{B_{\lam^{2}}}\|(\sum_{k\in\mathbb{Z}}|e^{it\Delta}g_{\square_{k}}|^{2})^{\frac{1}{2}}\|^{4}_{L^{4}_{t,x}(w_{B_{\lam^{2}}})})^{\frac{1}{4}}. 
    \end{equation*}
    Since $B_{\lam^{2}}$ are finitely overlapping with each other and $\sum_{B_{\lam^{2}}}w_{B_{\lam^2}} \lesssim 1$, 
    \begin{equation*}
        \|e^{it\Delta}g\|_{L^{4}_{t,x}(B_{1}(0,\lam^{2})\times\R)} \lesssim \|(\sum_{k\in\mathbb{Z}}|e^{it\Delta}g_{\square_{k}})^{\frac{1}{2}}\|_{L^{4}_{t,x}(w_{B_{1}(0,\lam^{2})}\times\R)}
    \end{equation*}
    holds. Rescaling again, we have
    \begin{equation*}
        \|e^{it\Delta}f\|_{L^{4}_{t,x}(I\times\R)}\lesssim\|(\sum_{k\in\Z}|e^{it\Delta}\square_{k}f|^{2})^{\frac{1}{2}}\|_{L^{4}_{t,x}(w_{I}\times\R)}. 
    \end{equation*}
    Next, we divide the summation on the left-hand side of this inequality by $\sum_{k\in\Z} = \sum_{0\leq j\leq 3}\sum_{k \in 4\Z + j}$. Then we have 
    \begin{align*}
        &\|(\sum_{k\in\Z}|e^{it\Delta}\square_{k}f|^{2})^{\frac{1}{2}}\|_{L^{4}_{t,x}(w_{I}\times\R)}\\
        &= \|(\sum_{0\leq j\leq 3}\sum_{k \in 4\Z + j}|e^{it\Delta}\square_{k}f|^{2})^{\frac{1}{2}}\|_{L^{4}_{t,x}(w_{I}\times\R)}\\
        &\leq \sum_{0\leq j\leq 3}\|(\sum_{k \in 4\Z + j}|e^{it\Delta}\square_{k}f|^{2})^{\frac{1}{2}}\|_{L^{4}_{t,x}(w_{I}\times\R)}\\
        &= \sum_{0\leq j\leq 3}(\int_{\R^{2}}\sum_{k,k'\in 4\Z + j}|e^{it\Delta}\square_{k}f|^{2}|e^{it\Delta}\square_{k'}f|^{2}w_{I}(t)dtdx)^{\frac{1}{4}}\\
        &= \sum_{0\leq j\leq 3}(\sum_{k,k'\in 4\Z + j}\|[e^{it\Delta}\square_{k}f][e^{it\Delta}\square_{k'}f]\|^{2}_{L^{2}_{t,x}(w_{I}\times\R)})^{\frac{1}{4}}\\
        &\leq \sum_{0\leq j\leq 3}(\sum_{\substack{k,k'\in 4\Z + j \\ k = k'}}\|e^{it\Delta}\square_{k}f\|^{4}_{L^{4}_{t,x}(w_{I}\times\R)})^{\frac{1}{4}} + \sum_{0\leq j\leq 3}(\sum_{\substack{k,k'\in 4\Z + j \\ k \neq k'}}\|[e^{it\Delta}\square_{k}f][e^{it\Delta}\square_{k'}f]\|^{2}_{L^{2}_{t,x}(w_{I}\times\R)})^{\frac{1}{4}}\\
        &\lesssim \|f\|_{M_{4,4}(\R)} + \sum_{0\leq j\leq 3}(\sum_{\substack{k,k'\in 4\Z + j \\ k \neq k'}}\|[e^{it\Delta}\square_{k}f][e^{it\Delta}\square_{k'}f]\|^{2}_{L^{2}_{t,x}(w_{I}\times\R)})^{\frac{1}{4}}
    \end{align*}
    here we used Proposition \ref{prop:unimodular} in the last inequality. 
    Since $k, k' = 4n + j$ for $j = 0,1,2,3$ and $k\neq k'$, we have $|k - k'| > 3$. Thus, we can apply Proposition \ref{thm:bilinear} to the second term on the right-hand side. It holds that 
    \begin{align*}
        &\|(\sum_{k\in\Z}|e^{it\Delta}\square_{k}f|^{2})^{\frac{1}{2}}\|_{L^{4}_{t,x}(w_{I}\times\R)}\\
        &\lesssim \|f\|_{M_{4,4}(\R)} + \sum_{0\leq j\leq 3}(\sum_{\substack{k,k'\in 4\Z + j \\ k \neq k'}}\ev{k - k'}^{-1}\|\square_{k}f\|^{2}_{L^{2}_{x}(\R)}\|\square_{k'}f\|^{2}_{L^{2}_{x}(\R)})^{\frac{1}{4}}\\
        &\lesssim \|f\|_{M_{4,4}(\R)} + (\sum_{k,k'\in\Z}\ev{k - k'}^{-1}\|\square_{k}f\|^{2}_{L^{2}_{x}(\R)}\|\square_{k'}f\|^{2}_{L^{2}_{x}(\R)})^{\frac{1}{4}}. 
    \end{align*}
    Therefore, applying the Young inequality, for any $p\in [1, 2)$, 
    \begin{equation*}
        \|(\sum_{k\in\Z}|e^{it\Delta}\square_{k}f|^{2})^{\frac{1}{2}}\|_{L^{4}_{t,x}(w_{I}\times\R)} \lesssim \|f\|_{M_{4,4}(\R)} + \|f\|_{M_{2,2p}(\R)}\\
        \lesssim \|f\|_{M_{2,2p}(\R)}
    \end{equation*}
    holds. This shows (\ref{eq:presmoothing}). 
\end{proof}

\begin{proof}[Proof of Theorem \ref{thm:loocalsmoothing}]
    Let $\{\Delta_{j}\}_{j\geq 0}$ be the Littlewood-Paley decomposition. By (\ref{eq:presmoothing}), we have
    \begin{equation*}
        \|e^{it\Delta}\Delta_{j}f\|_{L^{4}_{t,x}(I\times\R)} \lesssim \|\Delta_{j}f\|_{M_{2,4-\varepsilon}}
    \end{equation*}
    for each $j\in\mathbb{N}$. Thus, from the H\"{o}lder inequality and Proposition \ref{thm:littlewoodpaley}, we obtain
    \begin{align*}
        \|e^{it\Delta}f\|_{L^{4}_{t,x}(I\times\R)}
        &\lesssim \sum_{j}\|e^{it\Delta}\Delta_{j}f\|_{L^{4}_{t,x}(I\times\R)}\\
        &\lesssim \sum_{j}\|\Delta_{j}f\|_{M_{2,q}(\R)}\\
        &\lesssim_{\varepsilon} (\sum_{j}2^{\varepsilon j}\|\Delta_{j}f\|^{q}_{M_{2,q}(\R)})^{\frac{1}{q}}\\
        &\sim \|f\|_{M^{\varepsilon}_{2, q}(\R)}
    \end{align*}
    for any $\varepsilon > 0$ and for any $q\in[1,4)$. This completes the proof. 
\end{proof}

\section{Proof of Theorem \ref{thm:randomized}}

To show our last result, we use the following probabilistic estimate.
\begin{lem}[Lemma 3.1 in {\cite{burq2008random}}]
    Let $\{g_{n}\}_{n \in \Z^d}$ be a sequence with the assumption (\ref{assumption:randomization}) Then, there exists $C>0$ such that 
    \begin{equation}
        \|\sum_{k\in\Z^d}g_{k}c_{k}\|_{L^{p}(\Omega)} \leq C\sqrt{p}\|c_{k}\|_{\ell^{2}_{k}}
    \end{equation}
    for any $p\geq 2$ and $\{c_{k}\}_{k\in\Z^d}\subset \ell^{2}(\Z^d)$. \label{lem:randomized}
\end{lem}

\begin{proof}[Proof of Theorem \ref{thm:randomized}]
Let $(p,q)$ satisfy the assumption in Theorem \ref{thm:randomized}, and let $r > \mathrm{max}\{p,q\}$. By applying the Minkowski inequality and Lemma \ref{lem:randomized}, we have
\begin{align*}
    \|\|\sum_{k\in\Z^d}g_{k}(\omega)e^{it\Delta}\square_{k}f\|_{L^{p}_{t}L^{q}_{x}(\R^{d + 1})}\|_{L^{r}(\Omega)} 
    &\leq \|\|\sum_{k\in\Z^d}g_{k}(\omega)e^{it\Delta}\square_{k}f\|_{L^{r}(\Omega)}\|_{L^{p}_{t}L^{q}_{x}}\\
    &\leq C\sqrt{r}\|\|e^{it\Delta}\square_{k}f\|_{\ell^{2}_{k}}\|_{L^{p}_{t}L^{q}_{x}}. 
\end{align*}
    Next, to use the orthogonal Strichartz estimate, we divide $\|\square_{k}e^{it\Delta}f\|_{\ell^{2}_{k}}$ as follows: Let 
    \begin{equation*}
    S_{d} = \{(\sigma_{1},\sigma_{2},\cdots, \sigma_{d})\thinspace ;\thinspace \sigma_{i} = 0 \thinspace\mathrm{or}\thinspace1\thinspace(i = 1,2,\cdots,d)\}.
    \end{equation*}
    Then, we have
    \begin{align*}
        \|e^{it\Delta}\square_{k}f\|^{2}_{\ell^{2}_{k}} &= \sum_{k\in\Z^d}|e^{it\Delta}\square_{k}f|^{2}\\
        &= \sum_{s\in S_{d}}\sum_{n \in 2\Z^d + s}|e^{it\Delta}\square_{n}f|^{2}. 
    \end{align*}
    Note that $|S_{d}| = 2^{d}$ and the family $\{e^{it\Delta}\square_{n}f\}_{n\in 2\Z^d + s}$ is an orthogonal system in $L^{2}(\R^d)$ for each $s\in S_{d}$. Hence, from using (\ref{eq:ogs}), we obtain 
    \begin{align*}
        \|\|\sum_{k\in\Z^d}g_{k}(\omega)e^{it\Delta}\square_{k}f\|_{L^{p}_{t}L^{q}_{x}(\R^{d + 1})}\|_{L^{r}(\Omega)}
        &\lesssim \sqrt{r}\|\|e^{it\Delta}\square_{k}f\|_{\ell^{2}_{k}}\|_{L^{p}_{t}L^{q}_{x}}\\
        &\leq \sqrt{r}\sum_{s\in S_{d}} \|(\sum_{\substack{n = 2l + s\\ l\in\Z^d}}|e^{it\Delta}\square_{n}f|^{2})^{\frac{1}{2}}\|_{L^{p}_{t}L^{q}_{x}}\\
        &\lesssim \sqrt{r} \|\|\square_{k}f\|_{L^2}\|_{\ell^{2\alpha}_{k}}\\
        &\lesssim \sqrt{r}\|f\|_{M_{2,2\alpha}}. 
    \end{align*}
    where $\alpha = \frac{2q}{q + 2}$. Thus, by the Chebyshev inequality, it follows that there exists $C' > 0$ such that for any $r > \mathrm{max}\{p,q\}$, 
    \begin{equation}
        P(\|e^{it\Delta}f^{(\omega)}\|_{L^{p}_{t}L^{q}_x} > \lambda) \leq \left(\frac{C' r^{\frac{1}{2}}\|f\|_{M_{2,2\alpha}}}{\lambda}\right)^{r}. \label{eq:cheby}
    \end{equation}
    We set $r = (\frac{\lambda}{C'e\|f\|_{M_{2,2\alpha}}})^{2}$. If $r > \mathrm{max}\{p,q\}$, from the inequality (\ref{eq:cheby}), we obtain 
    \begin{equation*}
        P(\|e^{it\Delta}f^{(\omega)}\|_{L^p_{t}L^{q}_{x}} > \lambda) < e^{-c\lambda^{2}\|f\|^{-2}_{M_{2,2\alpha}}}. 
    \end{equation*}
    On the other hand, if $r \leq \mathrm{max}\{p,q\}$, we have the following trivial estimate
    \begin{align*}
        P(\|e^{it\Delta}f^{(\omega)}\|_{L^{p}_{t}L^{q}_{x}} > \lambda) 
        &\leq 1\\
        &= e^{\mathrm{max}\{p,q\}}e^{-\mathrm{max}\{p,q\}}\\
        &\leq e^{\mathrm{max}\{p,q\}}e^{-c\lambda^{2}\|f\|^{-2}_{M_{2,2\alpha}}}. 
    \end{align*}
    Therefore, the inequality
    \begin{equation*}
        P(\|e^{it\Delta}f^{(\omega)}\|_{L^{p}_{t}L^{q}_{x}} > \lambda) \leq c_{1}e^{-c_{2}\lambda^{2}\|f\|^{-2}_{M_{2,2\alpha}}}
    \end{equation*}
    holds. Finally, we take $\lambda = \frac{\|f\|_{M_{2,2\alpha}}}{\sqrt{c_{2}}}\left(\log(\frac{1}{\varepsilon}) + \log c_{1}\right)^{\frac{1}{2}}$. Then, the desired estimate (\ref{eq:randomrefinedstr}) holds with probability at least $1 - \varepsilon$. 
\end{proof}

\section{Proof of Proposition \ref{prop:necessary for Demeter}}
In this section, we will discuss the optimality for the reverse square function estimate for slabs. To show the sharpness, we use the following lemma:
\begin{lem}
    Let $s\geq 0$ and $p = \frac{2 (d + 2)}{d}$. Suppose that the reverse square function estimate
    \begin{equation}
        \|F\|_{L^{p}(\R^{d + 1})} \lesssim R^{s}\|(\sum_{k\in\Z^{d}}|F_{\widehat{\square}_{k}}|^{2})^{\frac{1}{2}}\|_{L^{p}(\R^{d + 1})}\label{asm:reverse}
    \end{equation}
    holds for all $R\geq 1$ and $F\in\mathcal{S}(\R^{d+1})$ with $\suppm \widehat{F}\subset\mathcal{N}_{R^{-2}}(\mathbb{P}^{d})$. Then, it follows that 
    \begin{equation*}
        \|e^{it\Delta}f\|_{L^{p}_{t,x}(\R^{d + 1})} \lesssim_{\varepsilon} \|f\|_{M^{s + \varepsilon}_{2,\frac{2(d + 2)}{d + 1}}(\R^{d})}
    \end{equation*}
    holds. \label{lem:reverse implies strichartz}
\end{lem}

\begin{proof}
    Note that Lemma \ref{lem:sufficientpart} reveals that assumption (\ref{asm:reverse}) implies the inequality 
    \begin{equation}
        \|\E f\|_{L^{p}(B_{R^2})} \lesssim R^{s}\|(\sum_{k\in\Z^{d}}|\E f_{\square_{k}}|^{2})^{\frac{1}{2}}\|_{L^{p}(w_{B_{R^{2}}})}. \label{asm:reverse2}
    \end{equation}
    for all $\suppm\widehat{f}\subset B_{d}(0,R)$. 
    We first show the inequality
    \begin{equation}
        \|e^{it\Delta}f\|_{L^{p}_{t,x}(\R\times\R^d)} \lesssim_{d}\lam^{s}\|f\|_{M_{2, \frac{2(d + 2)}{d + 1}}(\R^d)}. \label{eq:lemglobalstrichartz}
    \end{equation}
    for all $\lam \geq 1$ and $\suppm\widehat{f}\subset B_{d}(0,\lam)$. 
    Rescaling $f$ so that $\mathrm{supp}\widehat{g}\subset B_{d}(0,1)$, we have
    \begin{equation*}
        \|e^{it\Delta}f\|_{L^{p}_{t,x}(\R\times\R^{d})} = \lam^{-\frac{d + 2}{p}}\|e^{it\Delta}g\|_{L^{p}_{t,x}(\R\times\R^{d})}
    \end{equation*}
    where $g(x) := \lam^{-d}f(\lam^{-1}x)$. We decompose $\R\times\R^d$ into finitely overlapping $(d + 1)$ dimensional balls $B_{\lam^2}$ with radius $\lam^2$, that is, 
    \begin{equation*}
        \R^{d + 1} = \bigcup B_{\lam^2}. 
    \end{equation*}
    Then, we can write 
    \begin{equation*}
        \|e^{it\Delta}g\|_{L^{p}_{t,x}(\R\times\R^d)} \lesssim \left(\sum_{B_{\lam^2}}\|e^{it\Delta}g\|^{p}_{L^{p}_{t,x}(B_{\lam^2})}\right)^{\frac{1}{p}}. 
    \end{equation*}
    Applying (\ref{asm:reverse2}), we obtain 
    \begin{equation*}
        \|e^{it\Delta}g\|_{L^{p}_{t,x}(B_{\lam^2})} \lesssim \lam^{s}\|(\sum_{k\in\Z^d}|e^{it\Delta}g_{\square_{k}}|^{2})^{\frac{1}{2}}\|_{L^{p}_{t,x}(w_{B_{\lam^2}})}. 
    \end{equation*}
    Thus, we have
    \begin{align*}
        \|e^{it\Delta}g\|_{L^{p}_{t,x}(\R\times\R^d} &\lesssim\lam^{s}\left(\bigcup_{B_{\lam^{2}}}\|(\sum_{k\in\Z^d}|e^{it\Delta}g_{\square_{k}}|^{2})^{\frac{1}{2}}\|^{p}_{L^{p}_{t,x}(w_{B_{\lam^{2}}})}\right)^{\frac{1}{p}}\\
        &\lesssim \lam^{s}\|(\sum_{k\in\Z^d}|e^{it\Delta}g_{\square_{k}}|^{2})^{\frac{1}{2}}\|_{L^{p}_{t,x}(\R\times\R^d)}. 
    \end{align*}
    By rescaling again, it follows that
    \begin{equation*}
        \|e^{it\Delta}f\|_{L^{p}_{t,x}(\R\times\R^d)} \lesssim \lam^{s}\|(\sum_{k\in\Z^d}|e^{it\Delta}\square_{k}f|^{2})^{\frac{1}{2}}\|_{L^{p}_{t,x}(\R\times\R^d)}. 
    \end{equation*}
    Set $S_{d}:= \{(\sigma_{1},\sigma_{2},\cdots,\sigma_{d})\in\Z^d\thinspace;\thinspace \sigma_{i} = 0\thinspace or\thinspace 1\thinspace for \thinspace each\thinspace i=1,2,\cdots,d\}$. Then, the sum $\sum_{k\in\Z^d}$ is divided as $\sum_{l\in S_{d}}\sum_{n\in 2\Z^d + l}$. Hence, we can apply the orthogonal Strichartz estimate (\ref{eq:ogs}):
    \begin{align*}
        \|(\sum_{k\in\Z^d}|e^{it\Delta}\square_{k}f|^{2})^{\frac{1}{2}}\|_{L^{p}_{t,x}(\R\times\R^d)} 
        &\leq \sum_{l\in S_{d}}\|(\sum_{n\in 2\Z^d + l}|e^{it\Delta}\square_{n}f|^{2})^{\frac{1}{2}}\|_{L^{p}_{t,x}(\R\times\R^d)}\\
        &\lesssim \sum_{l\in S_{d}}\|\|\square_{k}f\|_{L^{2}(\R^d)}\|_{\ell^{2\alpha}_{k\in\Z^{d}}}\\
        &\lesssim 2^{d}\|\|\square_{k}f\|_{L^{2}(\R^d)}\|_{\ell^{2\alpha}_{k\in\Z^{d}}}
    \end{align*}
    where $\alpha = \frac{d + 2}{d + 1}$. Therefore, we obtain
    \begin{equation*}
        \|e^{it\Delta}f\|_{L^{p}_{t,x}(\R\times\R^d)} \lesssim_{d}\lam^{s}\|f\|_{M_{2, 2\alpha}(\R^d)}. 
    \end{equation*}

    Let $\{\Delta_{j}\}_{j\in\mathbb{N}}$ be the Littlewood-Paley decomposition. By (\ref{eq:lemglobalstrichartz}) and the Littlewood-Paley characterization of modulation spaces, we have
    \begin{align*}
        \|e^{it\Delta}f\|_{L^{p}_{t,x}(\R\times\R^d)}\leq\sum_{j\in\mathbb{N}}\|e^{it\Delta}\Delta_{j}f\|_{L^{p}_{t,x}(\R\times\R^d)} &\lesssim \sum_{j\in\mathbb{N}}2^{js}\|\Delta_{j}f\|_{M_{2,2\alpha}}\\
        &\lesssim_{\varepsilon} (\sum_{j\in\mathbb{N}}2^{j(s + \varepsilon)}\|\Delta_{j}f\|^{2\alpha}_{M_{2,2\alpha}})^{\frac{1}{2\alpha}}\\
        &\sim \|f\|_{M^{s+\varepsilon}_{2,2\alpha}(\R^{d})}.
    \end{align*}
    This completes the proof. 
\end{proof}

Now, we are ready to prove Proposition \ref{prop:necessary for Demeter}. 
\begin{proof}[Proof of Proposition \ref{prop:necessary for Demeter}]
Suppose that the reverse square function estimate (\ref{eq:asmnecessary}) holds. Then, by Lemma \ref{lem:reverse implies strichartz}, we have
\begin{equation*}
    \|e^{it\Delta}f\|_{L^{p}_{t,x}(\R^{d+1})} \lesssim \|f\|_{M^{s+\varepsilon}_{2,\frac{2(d + 2)}{d + 1}}(\R^d)}. 
\end{equation*}
On the other hand, Proposition \ref{prop:schippanecessary} reveals that $s > \frac{d}{2} - \frac{d + 1}{p}$. This completes the proof. 
\end{proof}

\appendix
\section{A direct proof for the reverse square function estimate when $d = 1$}
For the reader's convenience, we present a direct proof of Proposition \ref{prop:cordoba-fefferman type est}. This proof is based on the same strategy as in Hickman and Vitturi's lecture notes. The similar idea can be found in Demeter's textbook \cite{demeter2020fourier}. 
 
Let $\etabr(t,x)\in\mathcal{S}(\R^2)$ satisfy 
\begin{align*}
\begin{cases} 
\suppm \F_{t,x}[\etabr^{\frac{1}{4}}] \subset B_{2}(0,R^{-2})\\
\etabr\gtrsim 1 \thinspace on \thinspace B_{2}(0,R^2) . 
\end{cases}
\end{align*} We calculate the left-hand side of (\ref{eq:cordoba-fefferman}) as
\begin{align*}
    \|\E f\|^{4}_{L^{4}_{t,x}}(\etabr) &= \int_{\R^{2}}|\sum_{k\in\Z}\E f_{\square_{k}}|^4 \etabr(t,x) dtdx\\
    &= \int_{\R^2}(|\sum_{k\in\Z}\E f_{\square_{k}}|^{2})^{2}\etabr(t,x)dtdx\\
    &= \int_{\R^2}(\sum_{k,\tilde{k}} \E f_{\square_{k}} \overline{\E f_{\square_{\tilde{k}}}})^{2}\etabr(t,x)dtdx\\
    &= \|\sum_{k,\tilde{k}} (\E f_{\square_{k}}\etabr^{\frac{1}{4}}) \overline{(\E f_{\square_{\tilde{k}}}\etabr^{\frac{1}{4}})}\|^{2}_{L^{2}_{t,x}(\R^2)}. 
\end{align*}
Since $\suppm \F_{t,x}[\etabr^{\frac{1}{4}}] \subset B_{2}(0, R^{-2})$, the space-time Fourier transform of $\E f_{\square_{k}}\etabr^{\frac{1}{4}}$ is supported on a $R^{-2}$-neighborhood of a paraboloid on a line segment with length $\lesssim R^{-1}$ i.e.
\begin{equation*}
\suppm \F_{t,x}[\E f_{\square_{k}}\etabr^{\frac{1}{4}}]\subset \{(\xi_{1}, \xi_{2})\thinspace ; \thinspace R^{-1}(k-1) \lesssim \xi_{1} \lesssim R^{-1}(k + 1),\thinspace  |\xi_{1}|^{2} - R^{-2} \leq \xi_{2} \leq  |\xi_{1}|^{2} + R^{-2}\}. 
\end{equation*}
Let $\tsquare_{k}\subset \R^{2}$ denote the support of the space-time Fourier transform of $\E f_{\square_{k}}\etabr^{\frac{1}{4}}$. Applying the Cauchy-Schwarz inequality, we obtain the following.
\begin{align*}
    |\sum_{\mathrm{dist}(\tsquare_{k}, \tsquare_{\tilde{k}}) \lesssim R^{-1}}\E f_{\square_{k}}\etabr^{\frac{1}{4}} \overline{\E f_{\square_{\tilde{k}}}\etabr^{\frac{1}{4}}}| &\lesssim (\sum_{k}|\E f_{\square_{k}}\etabr^{\frac{1}{4}}|^{2})^{\frac{1}{2}}(\sum_{k}|\sum_{\substack{\tilde{k}, \\ \mathrm{dist}(\tsquare_{k}, \tsquare_{\tilde{k}}) \lesssim R^{-1}}}\E f_{\square_{\tilde{k}}}\etabr^{\frac{1}{4}}|^{2})^{\frac{1}{2}}\\
    &\lesssim (\sum_{k}|\E f_{\square_{k}}\etabr^{\frac{1}{4}}|^{2})^{\frac{1}{2}}. 
\end{align*}
Thus, it is enough to consider the case where $\mathrm{dist}(\tsquare_{k}, \tsquare_{\tilde{k}}) \gtrsim R^{-1}$ and show
\begin{equation*}
    \|\sum_{\mathrm{dist}(\tsquare_{k}, \tsquare_{\tilde{k}}) \gtrsim R^{-1}}\E f_{\square_{k}}\etabr^{\frac{1}{4}} \overline{\E f_{\square_{\tilde{k}}}\etabr^{\frac{1}{4}}}\|^{2}_{L^{2}_{t,x}(\R^2)} \lesssim \|(\sum_{k}|\E f_{\square_{k}}|^{2})^{\frac{1}{2}}\|^{4}_{L^{4}_{t,x}(\etabr)}. 
\end{equation*}
Observe 
\begin{equation*}
    \suppm \F_{t,x}[\E f_{\square_{k}}\etabr^{\frac{1}{4}}] * \F_{t,x}[\overline{\E f_{\square_{\tilde{k}}}\etabr^{\frac{1}{4}}}] \subset \tsquare_{k} - \tsquare_{\tilde{k}}. 
\end{equation*}
Suppose that 
\begin{equation*}
    \tsquare_{k} - \tsquare_{\tilde{k}} \cap \tsquare_{j} - \tsquare_{\tilde{j}} \neq \emptyset. 
\end{equation*}
Then there exists $y_{l} = (y^{(1)}_{l}, y^{(2)}_{l}) \in \tsquare_{l}\subset\R^{2}\quad (l = k,\tilde{k},j,\tilde{j})$ such that
\begin{equation*}
    y_{k} - y_{\tilde{k}} = y_{j} - y_{\tilde{j}}. 
\end{equation*}
Since each slabs $\tsquare_{l}$ belongs to the $R^{-2}$-neighbourhood of the paraboloid, that is, 
\begin{equation*}
\tsquare_{l} \subset \{(\xi_{1}, \xi_{2})\thinspace ; \thinspace -1 -R^{-1} \leq \xi_{1} \leq 1 + R^{-1},\thinspace |\xi_{1}|^{2} - R^{-2} \leq \xi_{2} \leq  |\xi_{1}|^{2} + R^{-2}\}, 
\end{equation*}
for $l = k,\tilde{k}, j,\tilde{j}$ there exists $t_{l} \in [-1 - R^{-2}, 1 + R^{-2}]$ such that
\begin{equation*}
    y^{(1)}_{l} = t_{l}
\end{equation*}
and 
\begin{equation*}
    |y^{(2)}_{l} - t^{2}_{l}|\lesssim R^{-2}. 
\end{equation*}
Using these relations, we have the following.
\begin{align*}
    |(t^{2}_{k} - t^{2}_{\tilde{k}}) - (t^{2}_{j} - t^{2}_{\tilde{j}})| &\leq |(t^{2}_{k} - t^{2}_{\tilde{k}}) - (y^{(2)}_{k} - y^{(2)}_{\tilde{k}})| + |(t^{2}_{j} - t^{2}_{\tilde{j}}) - (y^{(2)}_{j} - y^{(2)}_{\tilde{j}})| \\
    &\lesssim R^{-2}
\end{align*}
and it follows that
\begin{align*}
    |t_{k} - t_{\tilde{k}}||(t_{k} + t_{\tilde{k}}) - (t_{j} + t_{\tilde{j}})|
    &= |(t^{2}_{k} - t^{2}_{\tilde{k}}) - (t^{2}_{j} - t^{2}_{\tilde{j}})|\\
    &\lesssim R^{-2}. 
\end{align*}
Since we assumed $\mathrm{dist}(\tsquare_{k},\tsquare_{\tilde{k}}) \gtrsim R^{-1}$, we have $|t_{k} - t_{\tilde{k}}| \gtrsim R^{-1}$. Thus, combining this and the above inequality, we get
\begin{equation*}
    |(t_{k} + t_{\tilde{k}}) - (t_{j} + t_{\tilde{j}})| = |(t_{k} - t_{j}) + (t_{\tilde{k}} - t_{\tilde{j}})| \lesssim R^{-1}. 
\end{equation*}
Note that $|(t_{k} - t_{j}) - (t_{\tilde{k}} - t_{\tilde{j}})| = |(t_{k} - t_{\tilde{k}}) - (t_{j} - t_{\tilde{j}})| =0$. Thus, by combining them and applying the triangle inequality, we have
\begin{equation*}
    |t_{k} - t_{j}|, \thinspace |t_{\tilde{k}} - t_{\tilde{j}}| \lesssim R^{-1}. 
\end{equation*}
Therefore, from these inequalities and $|t_{l}| \leq 1 + R^{-2} \leq 2$ for $l = k,\tilde{k}, j, \tilde{j}$, we obtain the following:
\begin{align*}
    |y_{k} - y_{j}| &\leq |y_{k} - (t_{k}, t^{2}_{k})| + |(t_{k}, t^{2}_{k}) - (t^{2}_{j},t^{2}_{j})| + |(t_{j}, t^{2}_{j}) - y_{j}|\\
    &\lesssim R^{-2} + R^{-1} + R^{-2} \lesssim R^{-1}, 
\end{align*}
and
\begin{equation*}
    |y_{\tilde{k}} - y_{\tilde{j}}| \lesssim R^{-1}.
\end{equation*}
Hence, given $\tsquare_{k}, \tsquare_{\tilde{k}}$, there are at most $O(1)$ choices of $\tsquare_{j}, \tsquare_{\tilde{j}}$ for which $\tsquare_{k} - \tsquare_{\tilde{k}} \cap \tsquare_{j} - \tsquare_{\tilde{j}} \neq \emptyset$. Consequently, we have \begin{align*}
    &\|\sum_{\mathrm{dist}(\tsquare_{k}, \tsquare_{\tilde{k}}) \gtrsim R^{-1}}\E f_{\square_{k}}\etabr^{\frac{1}{4}} \overline{\E f_{\square_{\tilde{k}}}\etabr^{\frac{1}{4}}}\|^{2}_{L^{2}_{t,x}(\R^2)}\\
    &= \int_{\R^2}(\sum_{\mathrm{dist}(\tsquare_{k}, \tsquare_{\tilde{k}}) \gtrsim R^{-1}}\E f_{\square_{k}}\etabr^{\frac{1}{4}} \overline{\E f_{\square_{\tilde{k}}}\etabr^{\frac{1}{4}}})^{2}dtdx\\
    &= \int_{\R^{2}}\sum_{(k,\tilde{k}), (j,\tilde{j})}(\E f_{\square_{k}}\etabr^{\frac{1}{4}} \overline{\E f_{\square_{\tilde{k}}}\etabr^{\frac{1}{4}}}) \cdot (\overline{\E f_{\square_{j}}\etabr^{\frac{1}{4}} \overline{\E f_{\square_{\tilde{j}}}\etabr^{\frac{1}{4}}}})dtdx\\
    &\lesssim \sum_{\mathrm{dist}(\tsquare_{k}, \tsquare_{\tilde{k}}) \gtrsim R^{-1}}\int_{\R^{2}} (\E f_{\square_{k}}\etabr^{\frac{1}{4}} \overline{\E f_{\square_{\tilde{k}}}\etabr^{\frac{1}{4}}})^{2} dtdx\\
    &\lesssim \int_{\R^2}(\sum_{k}|\E f_{\square_{k}}\etabr^{\frac{1}{4}}|^{2})^{2}dtdx \\
    &= \|(\sum_{k}|\E f_{\square_{k}}|^{2})^{\frac{1}{2}}\|^{4}_{L^{4}_{t,x}(\etabr)}. 
\end{align*}
Noting $\etabr \lesssim w_{B_{R^2}}$, we obtain the result.

\section*{Acknowledgment}
The author thanks Professor Mitsuru Sugimoto for many comments. He also thanks Dr. Naoto Shida for helpful suggestions on Lemma \ref{lem:sufficientpart}. This work was supported by Grant-in-Aid for JSPS Fellows No. 24KJ1228.

\bigskip

Graduate School of Mathematics, Nagoya University, Furocho, Chikusaku, Nagoya 464-8602, Japan

\textit{Email}: inami.kotaro.u2@s.mail.nagoya-u.ac.jp


\begin{thebibliography}{10}

    \bibitem{benyi2007unimodular}
    {\'A}.~B{\'e}nyi, K.~Gr{\"o}chenig, K.~A. Okoudjou, and L.~G. Rogers.
    \newblock Unimodular {F}ourier multipliers for modulation spaces.
    \newblock {\em J. Funct. Anal.}, 246(2):366--384, 2007.
    
    \bibitem{benyi2015wiener}
    {\'A}.~B{\'e}nyi, T.~Oh, and O.~Pocovnicu.
    \newblock Wiener randomization on unbounded domains and an application to almost sure well-posedness of {NLS}.
    \newblock {\em Excursions in Harmonic Analysis, Volume 4: The February Fourier Talks at the Norbert Wiener Center}, pages 3--25, 2015.
    
    \bibitem{Benyi-Okoudiou}
    A.~B\'enyi and K.~A. Okoudjou.
    \newblock {\em Modulation spaces---with applications to pseudodifferential operators and nonlinear {S}chr\"odinger equations}.
    \newblock Applied and Numerical Harmonic Analysis. Birkh\"auser/Springer, New York, [2020] \copyright 2020.
    
    \bibitem{bourgain1996invariant}
    J.~Bourgain.
    \newblock Invariant measures for the {2D}-defocusing nonlinear {S}chr{\"o}dinger equation.
    \newblock {\em Comm. Math. Phys.}, 176(2):421--445, 1996.
    
    \bibitem{burq2008random}
    N.~Burq and N.~Tzvetkov.
    \newblock Random data {C}auchy theory for supercritical wave equations. {I}. {L}ocal theory.
    \newblock {\em Invent. Math.}, 173:449--475, 2008.
    
    \bibitem{burq2008random2}
    N.~Burq and N.~Tzvetkov.
    \newblock Random data {C}auchy theory for supercritical wave equations. {II}. {A} global existence result.
    \newblock {\em Invent. Math.}, 173(3):477--496, 2008.
    
    \bibitem{chaichenets2020local}
    L.~Chaichenets, D.~Hundertmark, P.~C. Kunstmann, and N.~Pattakos.
    \newblock Local well-posedness for the nonlinear {S}chr{\"o}dinger equation in the intersection of modulation spaces {$M^{s}_{p.q}(\mathbb{R}^d)\cap M_{\infty, 1}(\mathbb{R}^d)$}.
    \newblock In {\em Mathematics of Wave Phenomena}, pages 89--107. Springer, 2020.
    
    \bibitem{chen2024smoothing}
    X.~Chen, Z.~Guo, M.~Shen, and L.~Yan.
    \newblock On smoothing estimates for {S}chr\"odinger equations on product spaces {$\mathbb{T}^m\times \mathbb{R}^n$}.
    \newblock {\em J. Funct. Anal.}, 286(4):110262, 2024.
    
    \bibitem{cordero-rodino}
    E.~Cordero and L.~Rodino.
    \newblock {\em Time-frequency analysis of operators}, volume~75 of {\em De Gruyter Studies in Mathematics}.
    \newblock De Gruyter, Berlin, [2020] \copyright 2020.
    
    \bibitem{cordoba1979}
    A.~C\'ordoba.
    \newblock A note on {B}ochner-{R}iesz operators.
    \newblock {\em Duke Math. J.}, 46(3):505--511, 1979.
    
    \bibitem{demeter2020fourier}
    C.~Demeter.
    \newblock {\em Fourier {R}estriction, {D}ecoupling and {A}pplications}, volume 184.
    \newblock Cambridge University Press, 2020.
    
    \bibitem{fefferman}
    C.~Fefferman.
    \newblock A note on spherical summation multipliers.
    \newblock {\em Israel J. Math.}, 15:44--52, 1973.
    
    \bibitem{feichtinger1983modulation}
    H.~G. Feichtinger.
    \newblock {\em Modulation spaces on locally compact abelian groups}.
    \newblock Universit{\"a}t Wien. Mathematisches Institut, 1983.
    
    \bibitem{frank2014strichartz}
    R.~L. Frank, M.~Lewin, E.~H. Lieb, and R.~Seiringer.
    \newblock Strichartz inequality for orthonormal functions.
    \newblock {\em J. Eur. Math. Soc.}, 16(7):1507--1526, 2014.
    
    \bibitem{frank2017restriction}
    R.~L. Frank and J.~Sabin.
    \newblock Restriction theorems for orthonormal functions, {S}trichartz inequalities, and uniform {S}obolev estimates.
    \newblock {\em Amer. J. Math.}, 139(6):1649--1691, 2017.
    
    \bibitem{gan2024small}
    S.~Gan.
    \newblock Small cap square function estimates.
    \newblock {\em J. Fourier Anal. Appl.}, 30(3):36, 2024.
    
    \bibitem{gan2022note}
    S.~Gan, C.~Oh, and S.~Wu.
    \newblock A note on local smoothing estimates for fractional {S}chr{\"o}dinger equations.
    \newblock {\em J. Funct. Anal.}, 283(5):109558, 2022.
    
    \bibitem{gao2022type}
    C.~Gao, J.~Li, and L.~Wang.
    \newblock A type of oscillatory integral operator and its applications.
    \newblock {\em Math. Z.}, 302(3):1551--1584, 2022.
    
    \bibitem{gao2022improved}
    C.~Gao, C.~Miao, and J.~Zheng.
    \newblock Improved local smoothing estimates for the fractional {S}chr\"odinger operator.
    \newblock {\em Bull. Lond. Math. Soc.}, 54(1):54--70, 2022.
    
    \bibitem{grafakos_modern}
    L.~Grafakos.
    \newblock {\em Modern {F}ourier analysis}, volume 250 of {\em Graduate Texts in Mathematics}.
    \newblock Springer, New York, second edition, 2009.
    
    \bibitem{gressman2021reversing}
    P.~T. Gressman, S.~Guo, L.~B. Pierce, J.~Roos, and P.-L. Yung.
    \newblock Reversing a philosophy: from counting to square functions and decoupling.
    \newblock {\em J. Geom. Anal.}, 31:7075--7095, 2021.
    
    \bibitem{grochenig2001foundations}
    K.~Gr{\"o}chenig.
    \newblock {\em Foundations of time-frequency analysis}.
    \newblock Springer Science \& Business Media, 2001.
    
    \bibitem{klaus2023}
    F.~Klaus.
    \newblock Wellposedness of {NLS} in modulation spaces.
    \newblock {\em J. Fourier Anal. Appl.}, 29(1):Paper No. 9, 37, 2023.
    
    \bibitem{lu2023local}
    Y.~Lu.
    \newblock Local smoothing estimates of fractional {S}chr\"odinger equations in {$\alpha $}-modulation spaces with some applications.
    \newblock {\em J. Evol. Equ.}, 23(2):38, 2023.
    
    \bibitem{akihiko1981some}
    A.~Miyachi.
    \newblock On some singular {F}ourier multipliers.
    \newblock {\em Journal of the Faculty of Science, the University of Tokyo. Sect. 1 A}, 28(2):267--315, 1981.
    
    \bibitem{rogers2008local}
    K.~M. Rogers.
    \newblock A local smoothing estimate for the {S}chr\"odinger equation.
    \newblock {\em Adv. Math.}, 219(6):2105--2122, 2008.
    
    \bibitem{schippa2022smoothing}
    R.~Schippa.
    \newblock On smoothing estimates in modulation spaces and the nonlinear {S}chr{\"o}dinger equation with slowly decaying initial data.
    \newblock {\em J. Funct. Anal.}, 282(5):109352, 2022.
    
    \bibitem{tao2010physical}
    T.~Tao.
    \newblock A physical space proof of the bilinear {S}trichartz and local smoothing estimates for the {S}chr{\"o}dinger equation.
    \newblock {\em UTN-FRA, Buenos Aires}, 2010.
    
    \bibitem{wang2007global}
    B.~Wang and H.~Hudzik.
    \newblock The global {C}auchy problem for the {NLS} and {NLKG} with small rough data.
    \newblock {\em J. Differential Equations}, 232(1):36--73, 2007.
    
    \bibitem{wang_harmonic}
    B.~Wang, Z.~Huo, C.~Hao, and Z.~Guo.
    \newblock {\em Harmonic analysis method for nonlinear evolution equations. {I}}.
    \newblock World Scientific Publishing Co. Pte. Ltd., Hackensack, NJ, 2011.
    
    \end{thebibliography}
\end{document}